\documentclass[12pt,a4paper]{article}
\usepackage[english]{babel}

\usepackage{pstricks}
\usepackage{pst-node}
\usepackage{amsmath,boxedminipage}
\usepackage{amssymb}
\usepackage{graphicx}
\usepackage{enumerate}
\usepackage{color}
\usepackage[text={15cm,24cm}]{geometry}
\usepackage[normalem]{ulem}
\usepackage[ansinew]{inputenc}
\usepackage{ulem}

\newenvironment{proof}{{\bf Proof}:\ }%
   {~\ \hfill $\Box$\vspace{0,5cm}}

    {~\ \hfill$\Box$\vspace{0,5cm}}
\newenvironment{ack}{\vskip5mm{\bf Acknowledgements:}}%

\newtheorem{theorem}{Theorem}[section]

\newtheorem{lemma}[theorem]{Lemma}
\newtheorem{proposition}[theorem]{Proposition}

\newtheorem{coro}[theorem]{Corollary}

\newtheorem{fact}{Fact}[section]

\graphicspath{{.}{graphics/}}
\newrgbcolor{lightlightlightgray}{0.9 0.9 0.9}

\numberwithin{equation}{section}

\begin{document}
\title{Partition of graphs with maximum degree ratio}

\author{V. Bouquet\footnotemark[1], F. Delbot\footnotemark[2],
C.\ Picouleau\footnotemark[1]%, S. Rovedakis\footnotemark[1]
}
\date{\today}

\def\thefootnote{\fnsymbol{footnote}}

\footnotetext[1]{ \noindent
Conservatoire National des Arts et M\'etiers, CEDRIC laboratory, Paris (France). Email: {\tt
valentin.bouquet@lecnam.fr,christophe.picouleau@cnam.fr}}
\footnotetext[2]{ \noindent
LIP6 laboratory, Paris (France). Email: {\tt
Francois.Delbot@lip6.fr}
}
\maketitle

\begin{abstract}
Given a graph $G$ and a non trivial partition $(V_1,V_2)$ of its vertex-set, the satisfaction of a vertex $v\in V_i$ is the ratio between the size of it's closed neighborhood in $V_i$ and the size of its closed neighborhood in $G$. The worst ratio over all the vertices defines the quality of the partition. We define $q(G)$ the degree ratio of a graph as the maximum of the worst ratio over all the non trivial partitions. We give bounds and exact values of $q(G)$ for some classes of graphs. We also show some complexity results for the associated optimization or decision problems.

 \vspace{0.2cm}
\noindent{\textbf{Keywords}\/}: vertex-partition, edge-cut, regular graph, bipartite graph, degree ratio, NP-complete.
\end{abstract}

\section{Introduction}\label{intro}
The problems of partitioning a graph under various constraints have been intensively studied. A natural problem  consists  to find a partition such that each element of the partition  maximizes the number of internal links. This problem is closely related to the problem of partitioning a graph into communities. One could think of partitioning the graph into cliques or subsets such that most of the vertices have more neighbors in their part than outside. We refer to Pontoizeau PhD. thesis \cite{Pontoizeau} for an overview of community detection in graphs.

On the contrary, there is the concept of external partition, that is splitting a graph in two parts such that a vertex has at least half of its neighbors in the other part. It is known that every graph has an external partition \cite{SatGrasurvey}. Following, we have the problem of finding an internal partition, that is splitting a graph in two parts such that a vertex has at least half of its neighbors in his own part. Clearly, not all graphs have an internal partition. One could just look at an odd clique.

These two concepts have appeared under different names. The external partition is also referred to co-satisfactory partition, unfriendly partition or offensive 2-alliance. The internal partition is also named satisfactory partition, friendly partition or defensive 2-alliance. Note that the problem of finding $k$-partitions under the previous constraints is referred to as offensive $k$-alliance and defensive $k$-alliance. For a complete overview of the different approach we refer to the work of Bazgan et al. \cite{SatGrasurvey}.

Morris \cite{Morris} took a peak at the more general problem. Given a parameter $q,\, 0<q<1,$  he partitioned the graph in two parts $A$ and $B$ such that every $v\in A$ (resp. $u\in B$) has at least $q\times d(v)$ of its neighbors in $A$ (resp. at least $(1-q)\times d(u)$ neighbors in $B$). He refers to such set as $q/(1-q)$-cohesive and Ban et al. \cite{Ban} refers to it as $q$-internal partition and $q$-external partition. \\

We define and motivate the partitioning  problem we will study in this article.
A community network can be modeled by a graph where each vertex correspond to a member and edges correspond to friendship. In a partition of the network the ratio between the number of friends of a member situated in his side of the partition over the total number of his friends (a member being friend with himself) is a measure of the satisfaction of the member. A larger part of the members are satisfied when this ratio is big. The worst ratio over all the members of the network is a measure of the quality of the partition. In the maximum degree ratio problem the objective is to find a partition that maximize the worst ratio over all the non trivial partitions of the network. Note that for the maximum degree ratio problem we are concerned with the closed neighborhoods of the vertices in contrary to the majority of the problems mentioned above.  \\

The paper is organized as follows. On the next section we give the main notations and definitions of graph theory we use throughout the paper, we formalize the maximum degree ratio problem and its decision version, and we give some results of graph partitioning  that will be used in some of our proofs. In Section \ref{bounds} we give lower and upper bounds on the degree ratio. We also show exact formulas for some classes of graphs including trees. Section \ref{regul} is dedicated to regular graphs. We give exact values for the degree ratio and show some NP-completeness results. In Section \ref{prod} we are interested in the cartesian product of two graphs. Lower bounds for the degree ratio and NP-completeness results will be proved. In Section \ref{conc} we conclude and give some open questions.

\section{Prelimiraries}\label{pref}
All graphs considered in this paper are finite, undirected, simple and with at least two vertices. Given a graph $G=(V(G),E(G))$, with vertex-set $V(G)$ and edge-set $E(G)$, the degree of a vertex $x\in V(G)$, that is the number of its neighbors, is denoted by $d_G(x)$ or simply $d(x)$ when the context is not ambiguous. The set of neighbors of a vertex $x$ is called its {\it open neighborhood}, or simply its {\it neighborhood}, it is denoted by $N(x)$ and its size is $d(x)$. The neighborhood of $x$ with $x$, $N(x)\cup\{x\}$, is its {\it closed neighborhood}, and its size is $d_G[x]=d_G(x)+1$ or simply $d[x]$. The minimum degree and the maximum degree of $G$ are denoted by $\delta(G),\Delta(G)$, respectively.   Given $X\subseteq V(G)$, the subgraph  induced by $X$ is denoted by $G[X]$, i.e. \ $V(G[X])=X,E(G[X])=\{xy\in E(G)\vert x,y\in X\}$. Given $S\subset V(G)$ we denote by $G-S$ the induced subgraph $G[V(G)\setminus S]$. When $S$ is a singleton $\{v\}$ then we denote by $G-v$ the subgraph $G[V(G)\setminus \{v\}]$.
The complete graph with $n$ vertices is $K_n$, also called a {\it clique}. The chordless cycle on $n$ vertices is $C_n$. The graph $C_3=K_3$ is a {\it triangle}. The {\it bipartite clique} with $m$ vertices in one side and $n$ vertices in the opposite side is denoted $K_{m,n}$. The {\it claw} is $K_{1,3}$. The {\it diamond} is the four clique without an edge, i.e. $K_4-e$. The graph $K_4-e+v$ is obtained from the diamond by adding a vertex $v$ adjacent to the two vertices of degree three in $K_4-e$. Let us define a $k$-triangle by $T_k,\,k\ge 1,$ the following graph. $V(T_k)=\{s,t,v_1,\ldots,v_k\}$ and $E(T_k)=\{st\}\cup\{v_is,v_it : 1\le i\le k\}$. Note that $T_1=C_3$. For a fixed set $\{H_1,\ldots,H_p\}$ of graphs, $G$ is {\it $(H_1,\ldots,H_p)$-free} if $G$ has no induced subgraph isomorphic to a graph in $\{H_1,\ldots,H_p\}$; if $p=1$ we may write $H_1$-free instead of $(H_1)$-free. The graph $\overline G$ denotes the {\it complement} of $G$.
Other  notations or definitions that are not given here can be found in \cite{Bondy}.

Given $G=(V(G),E(G))$, $(V_1,V_2)$ is a partition of $V(G)$ if $V_1\cap V_2=\emptyset,V_1\cup V_2=V(G)$ and $V_i\ne\emptyset, i=1,2$. We denote by ${\cal P}_G$ the set of all partitions of $G$. Given a partition $(V_1,V_2)$ and a vertex $v\in V_i,i=1,2$ let us denote $q^i_G(v)={d_{G[V_i]}[v]\over d_G[v]}$ its {\it degree ratio}. Given $G$, its {\it optimal degree ratio} is denoted by $q(G)=\max_{(V_1,V_2)\in {\cal P}}\{\min_{v\in V_i}\{q^i_G(v)\},i=1,2\}$. A partition $(V_1,V_2)$ is {\it optimal} when $\min_{v\in V(G)} \{q^i_G(v)\}=q(G)$. When $G=(V(G),E(G))$ is connected, a partition $(V_1,V_2)$ induces a {\it cut} of $G$, that is $C\in E(G)$ the subset of edges with exactly one endpoint in $V_i,\, i=1,2$, and $G'=(V(G),E(G)\setminus C)$ is disconnected. When $M\subseteq E$ is a cut and no edges in $M$ share a vertex then $M$ is called a {\it matching-cut}.\\
For the notations defined above the subscripts can be dropped when the context is unambiguous.\\

The decision problem associated with the optimal degree ratio is defined as:
\begin{center}
\begin{boxedminipage}{.99\textwidth}
\textsc{\sc Degree Ratio Partition} \\[2pt]
\begin{tabular}{ r p{0.8\textwidth}}
\textit{~~~~Instance:} &a graph $G$ and a positive rational   $q,1 \geq q> 0$.\\
\textit{Question:} &is $q(G)\geq q$ ?
\end{tabular}
\end{boxedminipage}
\end{center}

Note that Degree Ratio Partition is clearly in NP.\\

As ingredients of some of our proofs we will use the following results.
In \cite{Stiebitz} M. Stiebitz proves the following.
\begin{theorem}\label{Stiebitz}
Let $G$ be a graph and $f_1,f_2:V(G)\rightarrow N$ two functions. Assume that $d_G(x)\ge f_1(x)+f_2(x)+1$ for every vertex $x\in V(G)$. Then there is a partition $(V_1,V_2)$ of $G$ such that:

\begin{enumerate}
\item $d_{G[V_1]}(x)\ge f_1(x)$ for every vertex $x\in V_1$, and
\item $d_{G[V_2]}(x)\ge f_2(x)$ for every vertex $x\in V_2$.
\end{enumerate}
\end{theorem}

This result is strengthened by Jianfeng Hou et al. in the case of $(K_4-e+v)$-free graphs and $(K_3,C_8,K_{2,3})$-free graphs \cite{Hou}. Jie Ma et al. dealt with the case of $(C_4,K_4,diamond)$-free graphs \cite{JieMa}.

\begin{theorem}\label{Hou}\cite{Hou}
Let $G$ be a $(K_4-e+v)$-free graph and $f_1,f_2:V(G)\rightarrow N\setminus\{0\}$ two functions. Assume that $d_G(x)\ge f_1(x)+f_2(x)$ for every vertex $x\in V(G)$. Then there is a partition $(V_1,V_2)$ of $G$ such that:

\begin{enumerate}
\item $d_{G[V_1]}(x)\ge f_1(x)$ for every vertex $x\in V_1$, and
\item $d_{G[V_2]}(x)\ge f_2(x)$ for every vertex $x\in V_2$.
\end{enumerate}
\end{theorem}

\begin{theorem}\label{JieMa}\cite{Hou,JieMa}
Let $G$ be either a $(C_4, K_4,diamond)$-free graph or a $(K_3,C_8,K_{2,3})$-free graphs and $f_1,f_2:V(G)\rightarrow N\setminus\{0,1\}$ two functions. Assume that $d_G(x)\ge f_1(x)+f_2(x)-1$ for every vertex $x\in V(G)$. Then there is a partition $(V_1,V_2)$ of $G$ such that:

\begin{enumerate}
\item $d_{G[V_1]}(x)\ge f_1(x)$ for every vertex $x\in V_1$, and
\item $d_{G[V_2]}(x)\ge f_2(x)$ for every vertex $x\in V_2$.
\end{enumerate}
\end{theorem}

\section{Some bounds for $q(G)$}\label{bounds}
We give some bounds for the degree ratio of a graph. We also give exact values of $q(G)$ for particular graphs.

For a disconnected graph $G$ taking one component in $V_1$ and the remaining vertices in $V_2$ we obtain an optimal partition achieving $q(G)=1$. Hence the graphs we consider are assumed connected. Then any partition of $V(G)$ induces a non empty cut and $q(G)<1$.
\begin{lemma}\label{upbound}
Let $G$ be a connected graph, then
\begin{center}
  $q(G)\le\max_{uv\in E(G)}\min\{ {d(u)\over d[u]},{d(v)\over d[v]}\}$
\end{center}
\end{lemma}
\begin{proof}
Since $G$ is connected in any partition $(V_1,V_2)$ there exists $uv\in E(G)$ such that $u\in V _1,v\in V_2$. Then $q^1(u)\le{d(u)\over d[u]}$ and $q^2(v)\le{d(v)\over d[v]}$. The result follows.
\end{proof}

When $G$ is a tree we have the following.
\begin{coro}\label{tree}
If $G$ is a tree, then
\begin{center}
$q(G)=\max_{uv\in E(G)}\min\{ {d(u)\over d[u]},{d(v)\over d[v]}\}$
\end{center}
\end{coro}
\begin{proof}
By Lemma \ref{upbound}  taking an edge $uv$ maximizing $\min\{ {d(u)\over d[u]},{d(v)\over d[v]}\}$.
\end{proof}

We give some lower bounds.

\begin{proposition}\label{lowbound}
For every graph $G$ we have $q(G)\ge \min\{{1\over 2}, {1\over 2+{1\over p}}\}$ where \\ $2p=\min_{v\in V(G)}\{d_G(v) : d_G(v)\  {\rm is\  even}\}$. The bound is tight.
\end{proposition}
\begin{proof}
We use Theorem \ref{Stiebitz}. We define $f_i(v)=\lfloor {1\over 2}(d(v)-1)\rfloor,i=1,2$. Hence $d(v)\ge f_1(v)+f_2(v)+1$ and then there exists $(V_1,V_2)$ a partition of $V(G)$ such that $d_{G[V_i]}(v)\ge f_i(v),\, i=1,2$ for every vertex $v\in V(G)$. Hence ${{d_{G[V_i]}(v)+1}\over {d(v)}}\ge {{p+1}\over{2p+2}}={1\over 2},\, i=1,2$ when $d(v)=2p+1$, and ${{d_{G[V_i]}(v)+1}\over {d(v)}}\ge {{p}\over{2p+1}}={1\over 2+{1\over p}},\, i=1,2$ when $d(v)=2p$.

One can observe that $q(K_{2p})={1\over 2}$ and $q(K_{2p+1})={1\over 2+{1\over p}}$. It suffices to take a (almost) balanced partition.
\end{proof}

\begin{proposition}\label{lowboundC3}
For a $(K_4-e+v)$-free graph $G$ we have $q(G)\ge{1\over 2}$. The bound is tight.
\end{proposition}
\begin{proof}
We use Theorem \ref{Hou}. We define $f_i(v)=\lfloor {1\over 2}d(v)\rfloor,\,i=1,2$. Hence $d(v)\ge f_1(v)+f_2(v)$ and then there exists $(V_1,V_2)$ a partition of $V(G)$ such that $d_{G[V_i]}(v)\ge f_i(v),\, i=1,2$ for every vertex $v\in V$. Hence, ${{d_{G[V_i]}(v)+1}\over {d(v)}}\ge {{p+1}\over{2p+2}}={1\over 2},\, i=1,2$ when $d(v)=2p+1$, and ${{d_{G[V_i]}(v)+1}\over {d(v)}}\ge {{p+1}\over{2p+1}}> {1\over 2},\, i=1,2$ when $d(v)=2p$.

One can observe that $q(K_{2p+1,2p+1})={1\over 2}$. It suffices to take an almost balanced partition.
\end{proof}

\begin{proposition}\label{lowboundC4free}
For every graph $G$ that is either $(C_4,K_4,diamond)$-free or \\ $(K_3,C_8,K_{2,3})$-free, we have $q(G)>{1\over 2}$.
\end{proposition}
\begin{proof}
We use Theorem \ref{JieMa}.
We define $f_i(v)=\lceil{1\over 2}d(v)\rceil,i=1,2$. Hence $d(v)\ge f_1(v)+f_2(v)-1$ and then there exists $(V_1,V_2)$ a partition of $V(G)$ such that $d_{G[V_i]}(v)\ge f_i(v),i=1,2$ for every vertex $v\in V(G)$. Hence, ${{d_{G[V_i]}(v)+1}\over {d(v)}}\ge {{p+2}\over{2p+2}}>{1\over 2},i=1,2$ when $d(v)=2p+1$, and ${{d_{G[V_i]}(v)+1}\over {d(v)}}\ge {{p+1}\over{2p+1}}> {1\over 2},i=1,2$ when $d(v)=2p$.
\end{proof}

We give some exact values.

\begin{fact}\label{ktriangle}
For a $k$-triangle we have $q(T_k)=(\lfloor{k\over 2}\rfloor+1){1\over k+2}$.
\end{fact}
\begin{proof}
Clearly, if $s,t\in V_i$ then there exists $v\in V_j,j\ne i,$ such that $q^j(v)={1\over 3}$. Now from $V_1=\{s,v_1,\ldots,v_{\lfloor {k\over 2}\rfloor}\}, V_2=\{t,v_{\lfloor {k\over 2}\rfloor+1},\ldots,v_k\}$ we obtain $q(T_k)=(\lfloor{k\over 2}\rfloor+1){1\over k+2}$.
\end{proof}

Now we show the following.
\begin{proposition}\label{C3}
$q(G)={1\over 3}$ if and only if $G=C_3$. When $G\ne C_3$ then $q(G)\ge{2\over 5}$.
\end{proposition}
\begin{proof}
Clearly $q(C_3)={1\over 3}$. Let $G\ne C_3$. By Proposition \ref{lowbound} $q(G) > {1\over 3}$ when $\delta(G)>2$. If $G$ is not connected then $q(G)=1$. If $G$ has a leaf then $q(G)\ge{1\over 2}$.  Assume that $uv$ is a cut-edge and let $C_1,C_2$ be the two components of $G-uv$. Let $(V(C_1),V(C_2))$ the corresponding partition of $V(G)$. We have $q(u)={d(u)\over d[u]}\ge {2\over 3}$, $q(v)={d(v)\over d[v]}\ge {2\over 3}$, so $q(G)\ge {2\over 3}$. Now assume that $v$ is a cut-vertex. Let $C$ be a component of $G-v$ where $v$ has the least number of neighbors. Let $V_1=C,\, V_2=V(G)\setminus C$. For each vertex $u$ of $V_1$ we have $q^1(u)={d(u)\over d[u]}\ge {2\over 3}$. For each vertex $u$ of $V_2\setminus \{v\}$ we have $q^2(u)=1$ and for the cut-vertex $v$ we have $q^2(v)> {1\over 2}$. Thus we can assume that $G$ is connected, $\delta(G)=2$, and $G$ has no cut-edge nor cut-vertex and $G$ has at least four vertices.

The  proof is by induction on $\vert V(G)\vert$. Let ${\cal P}'\subseteq{\cal P}$ be the partitions of $G$ obtained as follows: $f_i(v)=p,\, i=1,2$ when $d(v)=2p+1$; $f_1(v)= p,\, f_2(v)=p-1$ when $d(v)=2p$. By Theorem \ref{Stiebitz} there exists $(V_1,V_ 2)\in \cal P$ such that $\vert N(v)\cap V_i\vert+1\ge f_i(v)+1,\, i=1,2$ for every vertex $v\in V_i$. Thus $(V_1,V_ 2)\in \cal P'$. We prove that there exists $(V_1,V_2)\in {\cal P'}$ such that $q^i(v)\ge{2\over 5},i=1,2$.

Let $\vert V(G)\vert=4$. Since $G$ has no cut-edge, no leaf and $G\ne T_2$ then $G=K_4$ or $G=C_4$. In both  cases there exists $(V_1,V_2)\in {\cal P'}$ with $q^i(v)\ge{1\over 2}\ge{2\over 5},i=1,2$. (Since $G$ has a perfect matching  $M=\{uv,u'v'\}$  it suffices to take  $V_1=\{u,v\},\, V_2=\{u',v'\}$).

We assume that there exists $(V'_1,V'_2)\in {\cal P'}$ with $q^i(v)\ge{2\over 5},\, i=1,2$, for any graph $G'$ with $4\le\vert V(G')\vert \le n-1,\, n\ge 5$.

Let $v\in V(G)$ with $N(v)=\{s,t\}$. The first case is when $st\not\in E(G)$. Let $G'$ be the graph obtained from $G-v$ by adding the edge $st$. From the induction hypothesis there exists a partition $(V'_1,V'_2)$ of $V'$ such that $q(G')\ge{2\over 5}$. If $s,t\in V'_1$ then taking $V_1=V'_1\cup\{v\},V_2=V'_2$ we have $q(G)\ge q(G')\ge{2\over 5}$. Now assume that $s\in V_1,\, t\in V_2$. By the induction hypothesis $q(G' + st) \ge {2\over 5}$, so we deduce that $q(G) \geq {2\over 5}$.

The second case is when $st\in E(G)$.  From the induction hypothesis there exists a partition $(V'_1,V'_2)$ of $G'=G-v$ such that $q(G')\ge{2\over 5}$. If $s,t\in V'_1$, respectively $s,t\in V'_2$, then taking $V_1=V'_1\cup\{v\},V_2=V'_2$, respectively $V_1=V_1',V_2=V'_2\cup\{v\}$, we have $q(G)\ge q(G')\ge{2\over 5}$.

Now, without loss of generality let $s\in V'_1,t\in V'_2$. First, let $d_{G'}(s)=2p+1$.  Note that $p\ge 1$, else $s$ would be a cut-vertex. We have $\vert N_{G'}(s)\cap V'_1\vert+1\ge p+1$. Then ${{\vert N_{G'}(s)\cap V'_1\vert+1}\over{d_{G}(s)+1}}\ge {p+1\over 2p+3}\ge{2\over 5}$. Second, let $d_{G'}(s)=2p$. We have $\vert N_{G'}(s)\cap V'_1\vert+1\ge p+1$. Then ${{\vert N_{G'}(s)\cap V'_1\vert+1}\over{d_{G}[s]}}\ge {p+1\over 2p+2}\ge{1\over 2}$. Taking $V_1=V'_1,V_2=V'_2\cup\{v\}$ we obtain that $q(G)\ge  {2\over 5}$.
\end{proof}

\begin{figure}[htbp]
\begin{center}
\includegraphics[width=14cm, height=3cm, keepaspectratio=true]{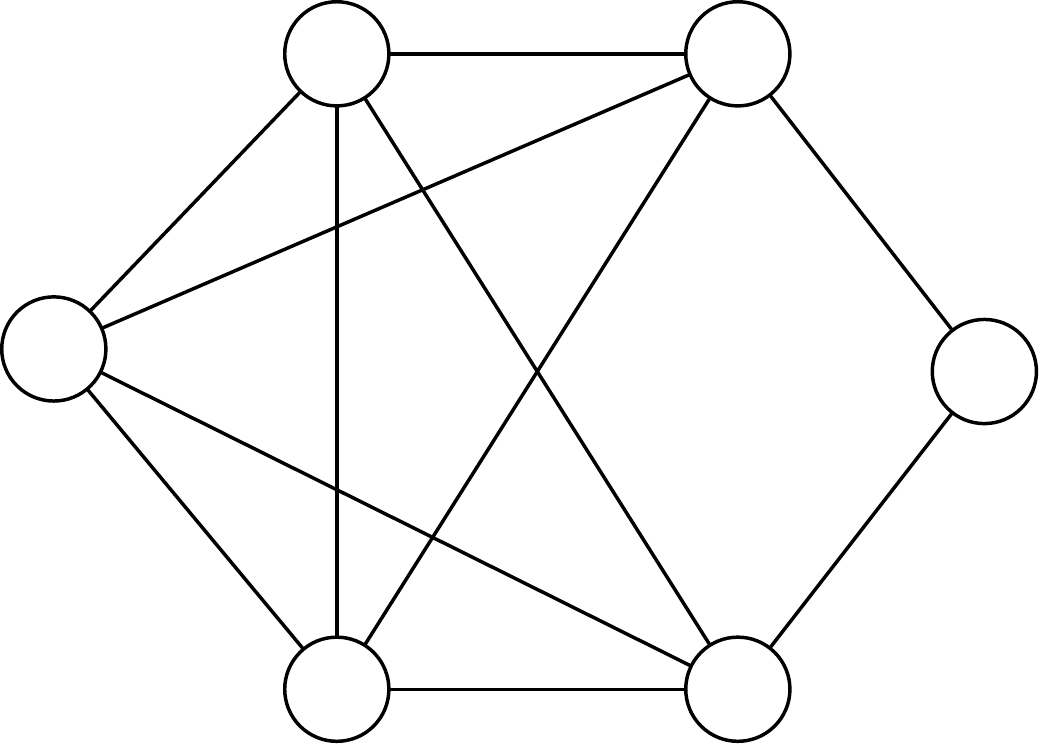}
\end{center}
\caption{The graph $G=\overline{K_2 \cup claw}$ with $q(G)={2\over 5}$.}
\label{gra25}
\end{figure}

\begin{proposition}\label{2/5D6}
 Let $G$ be a graph with $\Delta(G)\le 6$ and ${\cal F}=\{K_5,K_5-e,T_3,\overline{K_2 \cup claw}\}$. Then $q(G)={2\over 5}$ if and only if $G\in\cal F$. When $G\not\in{\cal F}\cup \{C_3\}$ then $q(G)\ge{3\over 7}$.
\end{proposition}
\begin{proof}
For each $G\in \cal F$ we have $q(G)={2\over 5}$. By Proposition \ref{C3} $q(G)\ge{2\over 5}$ when $G\ne C_3$. If $G$ has a cut-vertex then $q(G)>{1\over 2}$. When $G$ is $C_3$-free using Theorem \ref{Hou} with $f_1(v)=f_2(v)=\lfloor{d(v)\over 2}\rfloor$, $v\in V(G),$ we obtain $q(G)\ge{1\over 2}$.
So from now on, the graph $G$ we are concerned is biconnected and contains a triangle.\\

We define $(A,B)$ as a {\it good pair} if $A,B$ are two non empty vertex disjoint subsets of $V(G)$ such that
${{\vert N_{G}[v]\cap A\vert}\over{d_{G}[v]}}\ge{3\over 7}$ when $v\in A$ or ${{\vert N_{G}[v]\cap B\vert}\over{d_{G}[v]}}\ge{3\over 7}$ when $v\in B$.

We suppose that there exists a good pair $(A,B)$ for $G$.  If $A\cup B=V(G)$ then $(A,B)$ is a partition and $q(G)\ge{3\over 7}$. Now let $C=V(G)\setminus (A\cup B)$. If for every $v\in C$ we have ${{\vert N_{G}[v]\cap C\vert}\over{d_{G}[v]}}\ge{3\over 7}$ then $(A\cup B,C)$ is a partition and $q(G)\ge{3\over 7}$. Otherwise, there exists $v\in C$ such that ${{\vert N_{G}[v]\cap C\vert}\over{d_{G}[v]}}<{3\over 7}$. Since $d(v)\le 6$ we have $\vert N_{G}(v)\cap C\vert\le 1$. Hence $\max(\vert N_{G}(v)\cap A\vert,\vert N_{G}(v)\cap B\vert)\ge\lceil{1\over 2}(d(v)-1)\rceil$. Moreover ${{\lceil{1\over 2}(d(v)-1)\rceil+1}\over{d[v]}}\ge{3\over 7}$. Then we proceed as follows: when $\vert N_{G}(v)\cap A\vert\ge \vert N_{G}(v)\cap B\vert$ we set $A'=A\cup\{v\},B'=B$ otherwise $A'=A,B'=B\cup\{v\}$. Thus $(A',B')$ is a good pair. Iterating this procedure we obtain a partition of $G$ and $q(G)\ge {3\over 7}$.\\

The sequel of the proof consists to show that $G$ has a good pair. Let $T=\{a,b,c\}$ be a triangle of $G$. If $G-T$ contains a cycle $C$ then $(T,C)$ is a good pair. Now we can assume that $G-T$ is acyclic. Let $C$ be a connected component of $G-T$. Note that since $G$ is biconnected every leaf of $C$ has a neighbor in $T$. Moreover when $C$ consists of $v$, an isolated vertex, then $v$ has two neighbors in $T$.

The first case is when $C$ contains two distinct vertices $u,v$ such that $d_G[u],d_G[v]\le 4$. Let $P$ be the path between $u$ and $v$ in $C$. Since $\vert N[u]\cap P\vert=\vert N[v]\cap P\vert=2$ we have ${\vert N[u]\cap P\vert\over d_G[u]}={\vert N[v]\cap P\vert\over d_G[v]}\ge{2\over 4}>{3\over 7}$. Moreover, for every $w\in P,w\ne u,v,$ we have $\vert N[w]\cap P\vert=3$, and so ${\vert N[w]\cap P\vert\over d_G[w]}\ge{3\over 7}$. Hence $(T,P)$ is a good pair.

 From now on every component $C$ is a tree with at most one vertex $v$ with $d_G[v]\le 4$. We consider the case where $C$ has three leaves, say $u,v,w$. We can assume that $d[u],d[v]\ge 5$ and that $c$ is a neighbor of $w$. Let $P$ be the path from $v$ to $w$ in $C$. Since $P\cup \{c\}$ is a cycle, then $(P\cup \{c\}, \{a,b,u\})$ is a good pair. Thus from now on every component $C$ is either an isolated vertex or a path (for convenience, now, when we are referring to a path it is not an isolated vertex).

The second case is when there exists two paths $C,C'$. Let $u$ and $v$ be the two leaves of $C$, without loss of generality $d_G[v]=5$, so $v$ is complete to $T$ and $c$ is a neighbor of $u$. Let  $v'$ be a leaf of with $d[v']=5$. Then $(C\cup\{c\},\{a,b,v'\})$ is a good pair.

The third case is when $C$ is a path and there exists $w$ an isolated vertex in $G-T$. Let $u$ and $v$ be two leaves of $C$,  without loss of generality $d_G[v]=5$ and $c$ is a neighbor of $u$. Recall that $w$ has two neighbors in $T$. When $a,b\in N(w)$ then $(C\cup\{c\},\{a,b,w\})$ is a good pair. Otherwise without loss of generality $N(w)=\{a,c\}$. If $d_G[u]\leq 4$ then $(\{c,u,w\},\{a,b,v\})$ is a good pair. Else $d_G[u]=5$ and $(C\cup\{b\},\{a,c,w\})$ is a good pair.

The fourth case is when $C$ is the unique component of $G-T$ and $C$ is a path. Let $u$ and $v$ be two leaves of $C$. Without loss of generality $d_G[v]=5$ and $c$ is a neighbor of $u$. Firstly $C=u-v$. If $c$ is the unique neighbor of $u$ in $T$ then $(\{a,b\},\{c,u,v\})$ is a good pair; otherwise $G$ is either $K_5$ or $K_5-e$. Second, $C=v-\cdots-w-u$. When $d_G[u]\le 4$ we have $d_G[w]\ge5$ and thus $(\{c,u,w\}, \{a,b,v\})$ is a good pair. Otherwise $d_G[u]=5$. The first subcase is when $d_G[w]=3$. If $C=v-w-u$ then $G=\overline{K_2 \cup claw}$. Otherwise $w$ has a neighbor $s$ in $C$, $s\ne u$, such that $d_G[s]\ge 5$. Hence $s$ has a neighbor, say $c$, in $T$ and therefore $(\{c,s,u,w\}, \{a,b,v\})$ is a good pair. The second subcase is when $d_G[w]\ge 4$. So $w$ has a neighbor in $T$, say $c$. Then $(\{c,u,w\}, \{a,b,v\})$ is a good pair.

Eventually, every connected component of $G-T$ consists of an isolated vertex. Note that $3\le d[v]\le 4$.  Without loss of generality $6\ge d(a)\ge d(b)\ge d(c)\ge 2$. When $d(c)=2$ then $G=T_k$. Since $k\not\in\{1,3\}$, Fact \ref{ktriangle} implies $q(G)\ge {3\over 7}$. Now, let $d(c)=3$ and $v$ be the neighbor of $c$ in $G-T$. Then $(\{c, v\}, V(G)\setminus \{c,v\})$ is a good pair. Now,  $6\ge d(a)\ge d(b)\ge d(c)\geq 4$. If $G-T$ consists of exactly two vertices then $G=K_5-e\in\cal F$. So $G-T$ has at least three vertices. $c$ has two neighbors in $G-T$, say $u,v$. Since $d(a)\ge d(b)\ge d(c)$ there exists $w\in G-T,w\ne u,v,$ such that $a$ and $b$ are two neighbors of $w$. Then $(\{c,u,v\}, \{a,b,w\})$ is a good pair.
Hence $G$ has a good pair and the proof is completed.
\end{proof}

\section{Regular graphs}\label{regul}

We give the exact value of $q(G)$ when $G$ is a cubic graph. Then we show that Degree Ratio Partition is NP-complete in $k$-regular graphs, $k\geq 4$.\\

Given $G$ a connected $k$-regular graph then deciding if $q(G)\ge{k\over k+1}$ is equivalent to decide if $G$ has a matching-cut. From \cite{Chvatal,Gomes,LeRanderath} we know that deciding if a cubic graph has a matching-cut is polynomial (in fact among the cubic graphs only $K_3$ and $K_{3,3}$ have no matching-cut), and the problem becomes  NP-complete for $4$-regular graphs. The problem is also NP-complete when $G$ is bipartite and all the vertices of one side have a degree three and all the vertices of the other side have a degree four.

\begin{proposition}\label{cubic}
If $G$ is a connected $3$-regular graph distinct from $K_4$ or $K_{3,3}$ then $q(G)={3\over 4}$.
\end{proposition}
\begin{proof}
By the remark above we now that deciding if $q(G)\ge{3\over4}$ is equivalent to decide if $G$ has a matching-cut. Every cubic graph distinct from $K_3$ and $K_{3,3}$ has a matching-cut.\end{proof}

Chv\'atal in \cite{Chvatal} gives a polynomial time algorithm to compute a matching-cut for the graphs with $\Delta(G)\le 3$. So using this algorithm we can find an optimal partition of a cubic graph $G$  when $G$ is not $K_3$ or $K_{3,3}$.

\begin{proposition}\label{4reg}
Let $G$ be a connected $4$-regular graph.  If $G=K_5$ then $q(G)={2\over 5}$ else when $G$ contains a matching-cut then $q(G)={4\over 5}$ else $q(G)={3\over 5}$.
\end{proposition}
\begin{proof}
Since $G$ is connected we have $q(G)\le{4\over 5}$. We have $q(K_5)={2\over 5}$. Now let $G\ne K_5$.
Clearly $q(G)={4\over 5}$ if and only if $G$ has a matching-cut. Now let $G$ with no matching-cut.
When $G$ is triangle-free then by Theorem \ref{Hou} with $f_1(v)=f_2(v)=2$ for every vertex $v\in V(G)$ we obtain $q(G)\ge{3\over 5}$. Now $G$ contains a triangle $T=\{a,b,c\}$. When $G-T$ has a cycle $C$ then $(T,C)$ is a good pair (here $(A,B)$ is a good pair if ${{\vert N_{G}(v)\cap A\vert+1}\over{d_{G}[v]}}\ge{3\over 5}$ when $v\in A$ or ${{\vert N_{G}(v)\cap B\vert+1}\over{d_{G}[v]}}\ge{3\over 5}$ when $v\in B$). The arguments are the same as for the proof of Property \ref{2/5D6}. So each connected component of $G-T$ is a tree. If a component has at least two vertices it has two leaves which are connected to $a,b,c$. In this case $G-T=K_2$, so $G=K_5$ which is impossible. So each component is a unique vertex which is impossible since $G$ is $4$-regular.
\end{proof}

\begin{proposition}\label{bipNPc}
Let $G$ be a connected bipartite graph with $\delta(G)=3$ and $\Delta(G)=4$.  Deciding if $q(G)={3\over 4}$ is NP-complete.
\end{proposition}
\begin{proof}
From Le and Randerath \cite{LeRanderath} we know the following: deciding wether $G$ a bipartite graph such that all the vertices of one side of the bipartition have degree three and all the vertices of the other side have degree four, has a matching-cut is NP-complete. If $G$ has a matching-cut then $G$ has a partition $(V_1,V_2)$ where for any $v\in V_i,i=1,2$ at most one of its neighbor is in $V_j,j\ne i$. Since all the vertices of degree three are in the same side of $G$ we obtain $q(G)={3\over 4}$. Conversely, from the structure of $G$ we have $q(G)\le {3\over 4}$. If there exists a partition $(V_1,V_2)$ with $q(G)={3\over 4}$ then for any  $v\in V_i,i=1,2$ at most one of its neighbor is in $V_j,j\ne i$. Hence $G$ has a matching-cut.
\end{proof}

\begin{lemma}\label{kregbipcut}
For every fixed $k,k\ge 4,$ deciding wether a connected $k$-regular bipartite graph $G$ has a matching-cut is NP-complete.
\end{lemma}
\begin{proof}
The proof is by induction on $k$. Let $k=4$.
From \cite{Gomes} we know that deciding if $G$, a connected $4$-regular graph, has a matching-cut is NP-complete. From $G$ we build (in polynomial time) a connected $4$-regular bipartite graph $G'=(V(G'),E(G'))$ with $V(G')=V_1\cup V_2$ as follows: each vertex $v\in V(G)$ has two copies $v_1\in V_1,v_2\in V_2$; the edge set is $E(G')=\{v_1u_2,u_1v_2:uv\in E(G)\}$. Clearly $C\subseteq E(G)$ is a matching-cut in $G$ if and only if $C'=\{u_1v_2,u_2v_1:uv\in C\}$ is a matching-cut in $G'$.

Let $k>4$ and assume that deciding wether $G$ a connected $(k-1)$-regular bipartite graph has a matching-cut is NP-complete. From $G$ we build $G'$ a $k$-regular bipartite graph as above taking $E(G')=\{v_1u_2,u_1v_2:uv\in E(G)\}\cup\{v_1v_2:v\in V(G)\}$. Then $C\subseteq E$ is a matching-cut of $G$ if and only if $C'=\{u_1v_2,u_2v_1:uv\in C\}$ is a matching-cut of $G'$.
\end{proof}

By Theorem \ref{Hou} we have $q(G)\ge{{\lfloor{k\over 2}\rfloor+1}\over{k+1}}$ when $G$ is a $k$-regular bipartite graph.

\begin{proposition}\label{kregbipNPc}
Let $G$ be a connected $k$-regular bipartite graph, $k\ge 4$. Then deciding wether $q(G)={k\over {k+1}}$ is NP-complete.
\end{proposition}
\begin{proof}
 We use Lemma \ref{kregbipcut}. Let $k,k\ge 4$.  Let  $G$ be a $k$-regular bipartite graph. If $G$ has matching-cut then it has a partition $(V_1,V_2)$ where for any $v\in V_i,i=1,2$ at most one of its neighbor is in $V_j,j\ne i$. Then $q(G)={k\over {k+1}}$ (recall that $G$ is connected so $q(G)\le{k\over {k+1}}$). Conversely, if there exists a partition $(V_1,V_2)$ with $q(G)={k\over {k+1}}$ then for any  $v\in V_i,i=1,2$ at most one of its neighbor is in $V_j,j\ne i$. Hence $G$ has matching-cut.
\end{proof}

\section{Product of Graphs}\label{prod}
We give some lower bounds on the degree ratio $q(G)$ when $G$ is the product of two graphs. We also show some NP-completeness results. \\

The cartesian product $G\square H$ of two graphs $G$ and $H$ is the graph whose vertex set is $V(G)\times V(H)$. Two vertices $(g_1,h_1)$ and $(g_2,h_2)$ are adjacent in $G\square H$ if either $g_1=g_2$ and $h_1h_2$ is an edge in $H$ or $h_1=h_2$ and $g_1g_2$ is an edge in $G$. For a vertex $v$ of $G$, the subgraph of $G\square H$ induced by the set $\{(g,h)\mid h\in H\}$ is called an $H$-fiber and is denoted by $H_g$. In a same way, we define  $G_h$ a $G$-fiber. Clearly, all $G$-fibers and $H$-fibers are isomorphic to $G$ and $H$ respectively. In this section we consider only the cartesian product of finite graphs $G$ and $H$ such that $G$ and $H$ are connected with at least two vertices.

Since the matching-cuts plays an important role in the maximum degree ratio problem we give the following lemma which will be used later.
\begin{lemma}\label{prodmatchcut}
  For every pair of graphs $G$ and $H$, the cartesian product $G\square H$ has a matching-cut if only if $G$ or $H$ has a matching-cut.
\end{lemma}
\begin{proof}
First, suppose that $G$ has a matching-cut. Let $(V_1, V_2)$ be a partition of $G$ such that $E(V_1,V_2)$ is a matching-cut of $G$. We build the partition $(V_1',V_2')$ of $G\square H$ as follows: for every $v\in V_i$ we take the corresponding $H$-fiber $H_v$ in $V_i'$, $i=1,2$. Clearly $E(V_1',V_2')$ is a matching-cut of $G\square H$. The previous argument works if $H$ has a matching-cut.

Conversely, neither $G$ nor $H$ has a matching-cut. By contradiction let $(V_1', V_2')$ be a partition of $G\square H$ such that $E(V_1',V_2')$ is a matching-cut of $G\square H$. If there exists an $H$-fiber $H_u$ such that $H_u\not\subset V_i'$, $i=1,2$ then  $E(V_1'\cap H_u,V_2'\cap H_u)$ corresponds to a matching-cut of $H$. The same holds if there exists a $G$-fiber $G_u$ such that $G_u\not\subset V_i$,$i=1,2$. Yet if there is no $H$-fiber and no $G$-fiber that is cut by the partition $(V_1',V_2')$, then $E(V_1',V_2')$ is not a matching-cut of $G\square H$ which is a contradiction.
\end{proof}

Now we show two lower bounds for $q(G\square H)$.
\begin{lemma}\label{prodpart}
  Let $(V_1, V_2)$ and $(V_1',V_2')$ be, respectively, two partitions of $G$ and $G'=G\square H$ such that $V_i'=\{v' \mid v'\in H_v, v\in V_i\}$, $i=1,2$. Then $min_{v'\in V(G\square H)}(q^i_{G\square H}(v')) > min_{v\in V(G)}(q^i_G(v))$, $i=1,2$.
\end{lemma}
\begin{proof}
  Clearly for each vertex $v\in V(G)$ the corresponding $H$-fiber $H_v$ is such that $H_v\subset V_i'$ and for every vertex $v'\in H_v$ we have $d_{G'\setminus H_v}(v')=d_G(v)$, $d_{H_v}(v')\ge 1$. Hence for each vertex $v'\in H_v$, $q^i_{G'}(v')={d_{G[V_i]}(v) + d_{H_v}(v')+1\over d_G(v) + d_{H_v}(v')+1} > {d_{G[V_i]}(v)+1\over d_G(v)+1}={d_{G[V_i]}[v]\over d_G[v]}=q^i_G(v)$, $i=1,2$. Thus we have $min(q^i_{G\square H}(v')) > min(q^i_G(v))$.
\end{proof}

Since the cartesian product of graphs is commutative taking a partition yielding to $q(G)$ or $q(H)$, we obtain the following:

\begin{proposition} \label{prodmax}
  For every pair of graphs $G$ and $H$,
  \begin{center}
    $q(G\square H)> max(q(G), q(H))$
  \end{center}
\end{proposition}

\begin{proposition}\label{prodlowbound}
  For every pair of graphs $G$ and $H$,
  \begin{center}
    $q(G\square H)> {1\over 2}$
  \end{center}
\end{proposition}
\begin{proof}
  We use Theorem \ref{Stiebitz}. We define $f_i(v)=\lfloor {1\over 2}d(v)-1 \rfloor$, $i=1,2$. Hence $d(v)\ge f_1(v) + f_2(v) + 1$ and there exists a partition $(V_1, V_2)$ of $G$ such that $d_{G[V_i]}\ge f_i(v)$, $i=1,2$. By Proposition \ref{prodmax} we can assume $q(G),q(H)<{1\over 2}$, thus there exists a vertex $v$ of $G$ such that $q^i_G(v)<{1\over 2}$. Note that $q^i_G(v)<{1\over 2}$ only if $d_G(v)=2p, p\ge1$.

  Let $G'=G\square H$. We define a partition $(V_1',V_2')$ of $G'$ as followed: $V_1'=\{v' \mid v'\in H_v, v\in V_1\}$ and $V_2'=\{v' \mid v'\in H_v, v\in V_2\}$. Let $v$ be a vertex of $G$ such that $d_G(v)=2p$ and $q^i_G(v)={p\over 2p+1}<{1\over 2}$. For the corresponding $H$-fiber $H_v$ we have $d_{G'\setminus H_v}(v')=d_G(v)$, $d_{H_v}(v')\ge 1$, $H_v\subset V_i$, $i=1,2$ for every vertices $v'$ of $H_v$. Hence we have $q^i_{G'}(v')={p+d_{H_v}(v')\over 2p+d_{H_v}(v')+1}\ge {1\over 2}$ for every vertices $v'$ of $H_v$ with equality only when $d_{H_v}(v')=1$. But if $d_{H_v}(v')=1$ then $H$ has a leaf and hence $q(H)\ge {1\over 2}$. Thus by Proposition \ref{prodmax} we obtain $q(G\square H)> {1\over 2}$.\end{proof}

\begin{proposition}\label{prodkregtree}
  If $G$ is a $k$-regular graph and $H$ is a tree then,
  \begin{center}
    $q(G\square H)=\max_{uv\in E(H)}\min\{ {d(u)+k\over d[u]+k},{d(v)+k\over d[v]+k}\}$
  \end{center}
\end{proposition}
\begin{proof}
By Lemma \ref{upbound} we have $q(G\square H)\le \max_{uv\in E(H)}\min\{ {d(u)+k\over d[u]+k},{d(v)+k\over d[v]+k}\}$. Let $uv\in E(H)$ such that $\min\{ {d(u)\over d[u]},{d(v)\over d[v]}\}$ is maximum. The tree $H$ without $uv$ consists of two connected components $C_u,C_v$. Let $(V_1,V_2)$ be the partition of $G\square H$ such that $V_1=\cup_{w\in C_u}\{G_w\}$ and $V_2=\cup_{w\in C_v}\{G_w\}$ where $G_w$ are the $G$-fibers induced by the vertices $w$ of $H$. Let $M=\{st\in E(G\square H): s\in V_1,t\in V_2\}$. Since $uv$ is a disconnecting edge in $H$ then $M$ is a matching-cut of $G\square H$. Thus $(V_1,V_2)$ garanties that $q(G\square H)\ge\max_{uv\in E(H)}\min\{ {d(u)+k\over d[u]+k},{d(v)+k\over d[v]+k}\}$.
\end{proof}

\begin{proposition}\label{bipK2}
Let $G$ be a connected bipartite graph with $\Delta(G)=5$. Deciding if $q(G\square K_2)\ge{5\over 6}$ is NP-complete.
\end{proposition}
\begin{proof}
By Le and Randerath \cite{LeRanderath} we know the following: deciding if $G$ a bipartite graph such that all the vertices of one side of the bipartition have degree three and all the vertices of other side have degree four, has a matching-cut is NP-complete. From $G$ we obtain $G'$ another bipartite graph by pairing every vertex $v$  to a new vertex $v'$, its {\it   twin}, $v'$ being linked only to $v$ in $G'$. Thus $\Delta(G')=5$. Let $H=G'\square K_2$. Clearly $H$ is obtained in polynomial time from $G$. If $G$ has a matching-cut then $G$ has a partition $(V_1,V_2)$ where for any $v\in V_i,i=1,2$ at most one of its neighbor is in $V_j,j\ne i$. From $(V_1,V_2)$ we obtain a partition $(V_1',V_2')$ of $H$ as follows: first, for each vertex $v\in V_i$ its twin $v'$ is put to $V_i$, second,  $V_i'=\{w\mid w\in K_2{\rm \ fiber\  of\ }v, v\in V_i\},i=1,2$. Hence $q(H)\ge{5\over 6}$.

Conversely, assume that $q(H)\ge {5\over 6}$. Let $(V_1',V_2')$ be a partition of $H$ achieving $q(H)\ge {5\over 6}$. Since for each twin  $v'$ we have $d_H(v')=2$, then $v'$ is in the same side of the partition as its two neighbors (so its two neighbors are in the same side).
Hence  for any  $v\in V'_i,i=1,2$ at most one of its neighbor is in $V_j,j\ne i$. Hence $G$ has a matching-cut.
\end{proof}

\begin{proposition}\label{prodbibregfix}
 Let $k$ be an integer $k\ge 4$ and $H$ be a fixed $k'$-regular graph without matching-cut. Deciding if $q(G\square H)\ge {k+k'\over k+k'+1}$ where $G$ is bipartite $k$-regular is NP-complete.
\end{proposition}
\begin{proof}
By Lemma \ref{kregbipcut} we know that for any integer $k\ge 4$ deciding if a bipartite $k$-regular graph $\tilde G$ has a matching-cut is NP-complete. Given $(V_1,V_2)$  a partition of $\tilde G$ inducing a matching-cut, we construct a partition $(V_1',V_2')$ of $G'=G\square H$ as follows: for every vertex $v\in V_i$, we put the according $H$-fiber $H_v$ in $V_i'$, $i=1,2$. Clearly, for each vertex $w\in V_i'$, $q^i_{G'}(w)\ge {k+k'\over k+k'+1}$.

Conversely let $(V_1',V_2')$ be a partition of $G'=G\square H$ achieving $q(G')\ge {k+k'\over k+k'+1}$. Since $H$ has no matching-cut each fiber $H_v$ is such that $H_v\subset V_i'$. Then we construct a partition $(V_1,V_2)$ of $\tilde G$ as follows: for every vertex $v\in V(G)$, if $H_v\subset V_i'$, we put $v$ in $V_i$. Clearly $(V_1,V_2)$ satisfies $q({\tilde G})\ge {k\over k+1}$.\end{proof}

Note that for any $k$-regular graph $G$ and any graph $H$ such that $H$ has a matching-cut we have $q(G\square H)\ge {\Delta(G\square H)\over \Delta(G\square H)+1}$. From a matching-cut $M$ of $H$ it suffices to take the partition $(V'_1,V'_2)$ of $G\square H$ where each $H_v$ fiber is partitioned according to $M$.

\begin{proposition}\label{prodcub}
 Let $G$ and $H$ be two connected cubic graphs.  Then if $G,H\in\{K_4,K_{3,3}\}$ then $q(G\square H)={5\over 7}$ else
 $q(G\square H)={6\over 7}$.
\end{proposition}
\begin{proof}
Assume without loss of generality that $G\not\in\{ K_4,K_{3,3}\}$. Hence $G$ has a matching-cut and from the remark just above $q(G\square H)={6\over 7}$. Now $G,H\in\{ K_4,K_{3,3}\}$. By Lemma \ref{prodmatchcut} $G\square H$ has no matching-cut. Thus $q(G\square H)\le{5\over 7}$. Let $G=H=K_4$. Let $V(G)=\{a,b,c,d\}$ and $V(H)=\{1,2,3,4\}$. Taking the partition $(V_1,V_2)$ of $G\square H$ where $V_1=\{(a,i),(b,i) : 1\le i\le 4\}$ and $V_2=\{(c,i),(d,i) : 1\le i\le 4\}$ we have $q(G\square H)\ge{5\over 7}$. Let $G=H=K_{3,3}$. Let $V(G)=\{a,b,c,d,e,f\}$ with $ab\in E(G)$ and $V(H)=\{1,2,3,4,5,6\}$ with $E(H[\{1,2,3\}])=\emptyset$. Taking the partition $(V_1,V_2)$ of $G\square H$ where $V_1=\{(a,i),(b,i) : 1\le i\le 6\}$ and $V_2=V(G\square H)\setminus V_1$ we have $q(G\square H)\ge{5\over 7}$. Eventually, let $G=K_4$ with  $V(G)=\{a,b,c,d\}$ and $H=K_{3,3}$ with $V(H)=\{1,2,3,4,5,6\}$ and $E(H[\{1,2,3\}])=\emptyset$. Taking the partition $(V_1,V_2)$ of $G\square H$ where $V_1=\{(a,i),(b,i) : 1\le i\le 6\}$ and $V_2=V(G\square H)\setminus V_1$ we have $q(G\square H)\ge{5\over 7}$. Hence $q(K_4\square K_4)=q(K_{3,3}\square K_{3,3})=q(K_4\square K_{3,3})={5\over 7}$.
\end{proof}

\section{Conclusion}\label{conc}
For a graph $G$ we defined $q(G)$, a graph parameter, its maximum degree ratio. This parameter is a measure of the local quality of a partition of $G$. We proved several bounds on $q(G)$, we characterized the graphs achieving some particular values, $q(G)={1\over 3}$ or $q(G)={2\over 5}$, for instance. We proved several NP-completeness results for the case of regular graphs. We also proved lower bounds and complexity results when the graph $G$ is the cartesian product of two graphs.\\

Some open questions are left:

\begin{itemize}
\item Can we drop the condition $\Delta(G)\le 6$ in the Proposition \ref{2/5D6};
\item Characterize the graphs with $q(G)={3\over 7}$ (or other fixed values $q(G)={4\over 9}, q(G)={5\over 11},\ldots$);
\item Determine the complexity of Degree Ratio Partition for the following classes: Complement of bipartite, Split, Cograph.
\end{itemize}

The maximum degree ratio concerns a $2$-partition of a graph. A further direction of research could concern $k$-partitions, $k\ge 3$. We can also consider partitions with additional constraints:  each set must have the same cardinality, each set  must be connected.

\begin{ack}
The authors express their deep thanks to St\'ephane Rovedakis for helpful discussions.\end{ack}

\end{document}